\documentclass[11pt,twoside]{article} 

\usepackage{epsfig}
\usepackage[usenames]{color}
\usepackage{amsmath}
\usepackage{amssymb}
\usepackage{mathrsfs}
\usepackage{amsfonts}
\usepackage{amsthm}
\usepackage{lscape}
\usepackage{graphicx}
\usepackage{hyperref}
\hypersetup{colorlinks=false,pdftitle="Reverse entropy power inequalities",pdftex}

\theoremstyle{plain} 
\newtheorem{theorem}{Theorem}[section]
\newtheorem{proposition}[theorem]{Proposition}
\newtheorem{corollary}[theorem]{Corollary}
\newtheorem{lemma}[theorem]{Lemma}

\theoremstyle{remark} 
\newtheorem{remark}[theorem]{Remark}

\newcommand{\E}{{\bf E}}
\newcommand{\R}{\mathbb{R}}

\renewcommand{\P}{{\bf P}}

\renewcommand{\k}{{\kappa}}

\newcommand{\collS}{{\cal C}}
\newcommand{\nth}{\frac{1}{n}}
\newcommand{\setS}{s}

\newcommand{\be}{\begin{equation}}
\newcommand{\en}{\end{equation}}
\newcommand{\bee}{\begin{eqnarray*}}
\newcommand{\ene}{\end{eqnarray*}}
\def\ben{\begin{eqnarray*}}
\def\een{\end{eqnarray*}}

\numberwithin{equation}{section}

\begin{document}

\title{Reverse Brunn-Minkowski and reverse entropy power inequalities for  convex measures}
\author{Sergey Bobkov\thanks{S. Bobkov is with the
              School of Mathematics, University of Minnesota, 127 Vincent Hall, 
              206 Church St. S.E., Minneapolis, MN 55455 USA.
              Email: {\tt bobkov@math.umn.edu}. S.B. was supported in part by U.S. National Science Foundation grant DMS-1106530.}
 and Mokshay Madiman\thanks{
M. Madiman is with the
          Department of Statistics, Yale University, James Dwight Dana House,
          24 Hillhouse Ave, New Haven, CT 06511 USA.
          Email: {\tt mokshay.madiman@yale.edu}. M.M. was supported by a Junior Faculty Fellowship from Yale University and the
U.S. National Science Foundation CAREER grant DMS-1056996.
} 
}
\date{}
\maketitle

\begin{abstract}
We develop a reverse entropy power inequality for convex measures, 
which may be seen as an affine-geometric inverse of the entropy power inequality of Shannon and Stam. 
The specialization of this inequality to log-concave measures may be seen as a version of Milman's 
reverse Brunn-Minkowski inequality. The proof relies on a demonstration of new relationships 
between the entropy of high dimensional random vectors and the volume of convex bodies, 
and on a study of effective supports of convex measures, both of which are of independent interest, 
as well as on Milman's deep technology of $M$-ellipsoids and on certain information-theoretic inequalities. 
As a by-product, we also give a continuous analogue of some Pl\"unnecke-Ruzsa inequalities from
additive combinatorics.
\end{abstract}

\section{Introduction}
\label{sec:intro}
\setcounter{equation}{0}

The reverse Brunn-Minkowski inequality is a deep result in Convex Geometry 
discovered by V.~D.~Milman in the mid 1980s (cf. \cite{Mil86,Mil88:2,Mil88:1,Pis89:book}).
It states that, given two convex bodies $A$ and $B$ in $\R^n$, 
one can find linear volume preserving maps $u_i:\R^n \rightarrow \R^n$ 
$(i=1,2)$ such that with some absolute constant $C$
\be\label{eq:reverseBM}
\big|\widetilde A + \widetilde B\big|^{1/n} \leq C 
\left(|A|^{1/n} + |B|^{1/n}\right),
\en
where $\widetilde A = u_1(A)$, $\widetilde B = u_2(B)$, 
$\widetilde A + \widetilde B = 
\big\{x+y: x \in \widetilde A, \ y \in \widetilde B\big\}$
is the Minkowski sum, and where $|A|$ stands for the $n$-dimensional volume.
(Of course, one of these maps may be taken to be the identity operator.)
A similar inequality continues to hold for finitely many convex bodies with 
constants depending on the number of sets involved.

Note that the reverse inequality to \eqref{eq:reverseBM},
\be\label{eq:BM}
\big|\widetilde A + \widetilde B\big|^{1/n} \geq |A|^{1/n} + |B|^{1/n},
\en
holds true for any such $u_i$ by the usual Brunn-Minkowski inequality. 
Without loss of generality both relations may be written for convex
bodies with volume one, when \eqref{eq:reverseBM}--\eqref{eq:BM} take a simpler form
\be\label{eq:BM-norm}
2 \leq |A + \widetilde B|^{1/n} \leq 2C.
\en

Milman's inverse Brunn-Minkowski inequality has connections with 
high dimensional phenomena in Convex Geometry. For instance, it 
is known that proving Milman's inequality for convex bodies in isotropic 
position is equivalent to the hyperplane conjecture (\cite{BKM04}). 
It has also found a number of interesting extensions and applications 
(cf. \cite{KT05}, \cite{KM05}, \cite{AMO08}).

Our primary goal in this note is to develop an entropic generalization of 
the reverse Brunn-Minkowski inequality \eqref{eq:reverseBM}, 
which would involve arbitrary log-concave probability distributions 
rather than just uniform measures on compact convex sets. More generally, 
we consider convex (also called hyperbolic) measures, i.e., having 
densities of the form
\be\label{eq:cvx-def}
f(x) = V(x)^{-\beta}, \qquad x \in \R^n,
\en
where $V$ are positive convex functions on $\R^n$ and $\beta>n$ is
a given parameter. (To be precise, these are the densities
of the so-called $\k$-concave measures for $\k=(n-\beta)^{-1}$;
see Section~\ref{sec:cvx} for details.)
A secondary goal of this note is to 
develop a technology for going from entropy estimates to volume
estimates in convex geometry; this is developed in Section~\ref{sec:ent-vol},
and underlies the claim that our main result, stated purely in terms of
entropies, is a generalization of Milman's inverse Brunn-Minkowski inequality.

The afore-mentioned entropic generalization may be stated as an inverse of 
the entropy power inequality, in the same sense that Milman's inequality 
is an inverse of the Brunn-Minkowski inequality.
Given a random vector $X$ in $\R^n$ with density $f(x)$, introduce the 
entropy functional (or the differential entropy, or the Boltzmann--Shannon entropy),
$$
h(X) = - \int_{\R^n} f(x) \log f(x)\,dx,
$$
together with the entropy power
$$
H(X) = e^{2h(X)/n},
$$
provided that the integral exists in the Lebesgue sense. In particular, 
if $X$ is uniformly distributed in a convex body $A \subset \R^n$, we have
$$
h(X) = \log |A|, \qquad H(X) =|A|^{2/n}.
$$
These identities themselves suggest reviewing a number of results on  
volume relations in terms of the entropy, and also inspire one to find 
analogues of such relations for different classes of multidimensional 
probability distributions in the language of information theory. 

The entropy power inequality, due to Shannon and Stam 
(\cite{Sha48}, \cite{Sta59}, cf. also \cite{Cos85b}, \cite{Dem89}, \cite{Vil00} for a refinement when
one of the random vectors is normal, and \cite{ABBN04:1}, \cite{MB07} for other refinements), 
asserts that
\be\label{eq:EPI}
H(X+Y) \geq H(X) + H(Y),
\en
for any two independent random vectors $X$ and $Y$ in $\R^n$, for which
the entropy is defined. Although it is not directly equivalent to the 
Brunn-Minkowski inequality, it is very similar to it \cite{CC84}. 
For example, being restricted to normal random vectors
$X,Y$ with covariance matrices $R,S$, the inequality \eqref{eq:EPI} becomes 
Minkowski's inequality for determinants of positive definite matrices,
$$
{\rm det}^{1/n}(R+S) \geq {\rm det}^{1/n}(R) + {\rm det}^{1/n}(S).
$$
It includes the Brunn-Minkowski inequality for parallepipeds and therefore 
extends, by a simple bisection argument of Hadwiger-Ohmann (or in view of 
the infinitesimal character of the Brunn-Minkowski inequality),
to the class of all Borel measurable subsets of the Euclidean space.
Conversely, one may deduce the entropy power inequality as a consequence
of a Brunn-Minkowski inequality for restricted sums of sets \cite{SV96:epi, SV00}. 
Moreover, both the Brunn-Minkowski
and the entropy power inequalities can be given similar proofs as limiting cases of
Young's inequality for convolution with sharp constant \cite{DCT91}. 

In order to judge the sharpness of the entropy power inequality \eqref{eq:EPI}, we need to keep in mind
that the entropy is invariant under linear volume preserving transformation
of the space, i.e., $H(u(X)) = H(X)$ whenever $|{\rm det}(u)|=1$.
On the other hand, the left side of \eqref{eq:EPI} essentially depends
on ``positions'' of the distributions of $X$ and $Y$, in the sense that it
is sensitive to linear volume preserving transformation of either $X$ or $Y$. 
Therefore, to reverse this inequality, some transformation of these random vectors is needed.
Specifically, we have:

\begin{theorem}\label{thm:repi}
Fix $\beta_0 > 2$. Let $X$ and $Y$ be independent random vectors 
in $\R^n$ with densities of the form $\eqref{eq:cvx-def}$ with $\beta \geq \max\{\beta_0 n, 2n+1\}$. 
There exist linear volume preserving 
maps $u_i:\R^n \rightarrow \R^n$ such that 
\be\label{eq:repi}
H\big(\widetilde X + \widetilde Y\big)\, \leq\, C_{\beta_0}\, (H(X) + H(Y)),
\en
where $\widetilde X = u_1(X)$, $\widetilde Y = u_2(Y)$, and where
$C_{\beta_0}$ is a constant depending only on $\beta_0$.
\end{theorem}

For growing $\beta$, the families \eqref{eq:cvx-def} shrink, and we arrive in the limit 
as $\beta \rightarrow +\infty$ at the class of log-concave densities 
(which correspond to the class of log-concave measures). 
Recall that log-concavity of a non-negative function 
$f$ on $\R^n$ is also defined through the inequality
$$
f(tx+sy) \geq f(x)^t g(y)^s, \qquad x,y \in \R^n, \ \ t,s > 0, \ t+s=1.
$$
Such functions are supported and positive on some open convex sets
in $\R^n$, where $\log f$ are concave (and we define them to be zero
outside supporting sets).

Thus, by Theorem~\ref{thm:repi}, if $X$ and $Y$ are independent and have log-concave 
densities, then for some linear volume preserving maps 
$u_i:\R^n \rightarrow \R^n$,
\be\label{eq:repi-lc}
H\big(\widetilde X + \widetilde Y\big)\, \leq\, C\, (H(X) + H(Y)),
\en
where $C$ is an absolute constant. This statement for the log-concave 
case was announced by the authors in \cite{BM11:cras}.

As for the general case, it can be shown that there does not exist a finite universal constant 
such that a reverse entropy power inequality holds for the entire class of convex 
measures, so that some restriction on the range of convexity parameter $\beta$
as in Theorem~\ref{thm:repi} is necessary  (see Proposition~\ref{prop:counter}).
Nevertheless, it would be interesting to explore how the constants
in the inequality \eqref{eq:repi} may depend on the remaining values $\beta > n$.


Let us state an equivalent variant of Theorem~\ref{thm:repi} by involving maximum 
of the density,
$$
\|f\| = {\rm ess\, sup}_x\, f(x),
$$
and keeping the same notations.

\begin{theorem}\label{thm:repi2}
Fix $\beta_0 > 2$. Let $X$ and $Y$ be independent random vectors 
in $\R^n$ with densities $f$ and $g$ of the form $\eqref{eq:cvx-def}$, such that 
$\|f\| = \|g\| = 1$. If $\beta \geq \max\{\beta_0 n, 2n+1\}$, there exist 
linear volume preserving maps $u_i:\R^n \rightarrow \R^n$ such that 
\be\label{eq:repi-norm}
c_0 n\, \leq\, h\big(\widetilde X + \widetilde Y\big)\, \leq\, c_{\beta_0} n
\en
with some absolute constant $c_0>0$, and some
constant $c_{\beta_0}$ depending only on $\beta_0$.
\end{theorem}

Equivalently, with some $C_{\beta_0} > C_0 > 1$, we have
\be\label{eq:repi-norm2}
C_0 \leq H\big(\widetilde X + \widetilde Y\big) \leq C_{\beta_0}.
\en

Being restricted to random vectors $X$ and $Y$ that are uniformly 
distributed in convex bodies $A$ and $B$, the reverse entropy power 
inequality \eqref{eq:repi} is equivalent to Milman's theorem \eqref{eq:reverseBM} modulo 
an absolute factor, while the right inequality in \eqref{eq:repi-norm2} is
equivalent to the right inequality in \eqref{eq:BM-norm} in a similar sense
(under the assumption $|A|=|B|=1$).

This generalization is however not immediate and has to be clarified, 
because the distribution of $X+Y$ is not uniform in $A+B$. Nevertheless, 
it is ``almost'' uniform, so that $H(X+Y)$ is of the same order as 
$|A+B|^{2/n}$. As will be explained later on, if $X$ and $Y$ are independent 
and uniformly distributed in $A$ and $B$, we have
\be\label{eq:volsum}
\frac{1}{4}\
|A + B|^{2/n} \leq\, H(X+Y) \leq\ |A + B|^{2/n}.
\en
These bounds allow one to freely translate many volume 
relations into statements about entropy.

As for the left inequality in \eqref{eq:repi-norm}, it immediately follows from the 
entropy power inequality \eqref{eq:EPI}, which implies ``concavity''
of the entropy functional: 
$$
h\bigg(\frac{\widetilde X + \widetilde Y}{\sqrt{2}}\bigg) \geq 
\frac{h(\widetilde X) + h(\widetilde Y)}{2} = 
\frac{h(X) + h(Y)}{2} \geq 0,
$$
where on the last step the assumption $f,g \leq 1$ is used. Hence, one may
take $c_0 = \log \sqrt{2}$ in \eqref{eq:repi-norm} and $C_0 = 2$ in \eqref{eq:repi-norm2},
similarly to the left inequality in \eqref{eq:BM-norm}.

\vskip2mm
It should be noted that there are other (non-entropic) formulations of 
the reverse Brunn-Minkowski inequality. In their study of the geometry 
of log-concave functions B.~Klartag and V. D. Milman have recently proposed
a natural functional generalization of \eqref{eq:reverseBM} in terms of the Asplund product
$$
f \star g(x) = \sup_y \big[f(x-y) g(y)\big], \qquad x \in \R^n.
$$
They prove (cf. \cite[Theorem 1.3]{KM05}) that, given symmetric log-concave functions 
$f$ and $g$ on $\R^n$, satisfying $f(0)=g(0)=1$, there
exist linear volume preserving maps $u_i:\R^n \rightarrow \R^n$ such that
with some absolute constant $C$, 
\be\label{eq:km}
\bigg(\int \tilde f \star \tilde g(x)\ dx\bigg)^{\!1/n} \leq\, C\,
\bigg[\bigg(\int f(x)\,dx\bigg)^{\!1/n} + \bigg(\int g(x)\,dx\bigg)^{\!1/n}\,\bigg],
\en
where $\widetilde f(x) = f((u_1(x))$ and $\widetilde g(x) = g(u_2(x))$.
Indeed, on the indicator functions $f = 1_A$, $g = 1_B$, we have
$
\tilde f \star \tilde g = 1_{\widetilde A + \widetilde B}, 
$
so \eqref{eq:km} reduces exactly to \eqref{eq:reverseBM}.

The inequality \eqref{eq:km} is related to the log-concave variant \eqref{eq:repi-lc}
in Theorem~\ref{thm:repi}. However, the Asplund product behaves differently than
the usual convolution, especially for  densities that are not log-concave.
Anyhow, in the proof of Theorems 1.1--1.2 themselves, the convex body case
as in \eqref{eq:reverseBM} or \eqref{eq:BM-norm}, that is, Milman's theorem, will be a basic 
ingredient in our argument, together with a general ``submodularity''  property of the
entropy functional (cf. \cite{Mad08:itw}), which has recently appeared in information theory.

\vskip2mm
The paper is organized as follows.
In Section~\ref{sec:cvx} we recall Borell's hierarchy and characterization of convex
measures and discuss convexity properties of convolutions, which are prerequisites
for the rest of the paper.

Section~\ref{sec:ent-vol} introduces a new tool for going from 
entropy estimates to volume estimates in convex geometry. 
The key idea here is that for sufficiently ``convex'' probability measures
(i.e., $\k$-concave probability measures for positive $\k$, which 
necessarily have compact support), the entropy can be approximated
in some sense by the logarithm of the volume of the support set.
While the fact that the entropy of a probability measure on a compact set is
bounded from above by  the logarithm of the volume of the support is simple
and classical, the corresponding lower bound under convexity assumptions
is new. In Section~\ref{sec:ent-max}, the entropy of convex measures is related 
to the maximum of their densities (which is of course related to the volume of the support
in the special case of the uniform distribution on a set), and some corollaries are discussed. 

The case of negative $\k$ is considered in Section~\ref{sec:supp}. 
In this case, although the support set of a $\k$-concave probability measure
may not be bounded, it is nonetheless possible to define in some sense
an ``effective support'', which is bounded and whose volume is related to
the entropy of the measure. In this sense, the relation between entropy and
volume can be extended to general convex measures, and moreover,
this may be thought of as providing a reverse technology to go from 
volume estimates to entropy estimates in convex geometry by using the notion
of effective supports. Some refinements of these ideas, related to an
asymptotic equipartition property for log-concave measures, are
described in \cite{BM11:aop}. 

Next, in Section~\ref{sec:m-pos}, we turn to the notion of $M$-positions 
of convex bodies, first developed by V. Milman, and show using
the afore-mentioned effective support idea that such a notion can be
defined for convex measures.
Section~\ref{sec:submod} introduces into convex geometry
a submodularity result for the entropy of sums,
first developed in \cite{Mad08:itw}, and discusses some corollaries,
including the connection of $M$-positions of convex bodies with 
the reverse Brunn-Minkowski inequality, and continuous analogues
for volumes of convex bodies of the Pl\"unnecke-Ruzsa inequalities
that are well known in the discrete world of additive combinatorics.

Section~\ref{sec:lcpf} and ~\ref{sec:genpf} are devoted to completing
the proof of Theorem~\ref{thm:repi}-- the former for the log-concave case,
and the latter for the general convex measure case.
Finally, in Section~\ref{sec:disc}, we comment on the reverse entropy power
inequality \eqref{eq:repi-lc} for log-concave measures
in the case where the distributions of $X$ and $Y$ are isotropic.

\vspace{.1in}
\noindent{\bf Acknowledgments.}
We are grateful to an anonymous referee for several useful suggestions to improve
clarity of the paper, and to both him/her and K.~Ball for fleshing out our understanding 
of the history of Corollary~\ref{cor:lc-maxnorm} (discussed in Section~\ref{sec:ent-max}).



\section{Convex measures}
\label{sec:cvx}
\setcounter{equation}{0}

Here we recall basic definitions and the characterization of the so-called 
convex measures.

Given $-\infty \leq \k \leq 1$, a probability measure $\mu$ on $\R^n$ is 
called $\k$-concave, if it satisfies the Brunn-Minkowski-type inequality
\be\label{eq:kconc-def}
\mu \big (tA + (1-t)B \big ) \geq
\big [\,t\mu(A)^\k + (1-t)\mu(B)^\k\big ]^{1/\k}
\en
for all $t \in (0,1)$ and for all Borel measurable sets $A,B \subset \R^n$ 
with positive measure. When $\kappa = 0$, \eqref{eq:kconc-def} describes the class of 
log-concave measures which thus satisfy
$$
\mu \big (tA + (1-t)B \big ) \geq \mu(A)^t \mu(B)^{1-t}.
$$
In the absolutely continuous case,
the log-concavity of a measure is equivalent to the log-concavity of
its density (Pr\'ekopa's theorem \cite{Pre71}). When $\kappa = -\infty$, the right 
side is understood as $\min\{\mu(A), \mu(B)\}$. The inequality \eqref{eq:kconc-def} is 
getting stronger as the parameter $\k$ is increasing, so in the case 
$\k = -\infty$ we obtain the largest class, whose members are called convex 
or hyperbolic probability measures. 

For general $\k$'s, the family of $\k$-concave measures was introduced and 
studied by C. Borell \cite{Bor74,Bor75a} who gave the following characterization, which 
we state below in the absolutely-continuous case. In this case
necessarily $\k \leq 1/n$. See also \cite{BL76a}.

\begin{proposition}\label{prop:cvx-meas}
An absolutely continuous probability measure 
$\mu$ on $\R^n$ is $\k$-concave, where $-\infty \leq \k \leq 1/n$, if and 
only if $\mu$ is supported on an open convex set $\Omega \subset \R^n$, 
where it has a positive $\tilde\k$-concave density $f$, that is, satisfying
\be\label{eq:knconc-def}
f(tx + (1-t)y) \geq
\big [\,t f(x)^{\tilde\k} + (1-t) f(y)^{\tilde\k} \big ]^{1/\tilde\k}
\en
for all $t \in (0,1)$ and $x,y \in \Omega$.
\end{proposition}

Here and below we put
$$
\tilde\k = \frac{\k}{1 - n\k}, \qquad \beta = \frac{1}{|\tilde\k|}.
$$
Thus, $\mu$ is $\k$-concave if and only if $f$ is $\tilde\k$-concave.

If $\k\in (0,1/n)$, then $\tilde\k > 0$ and $\beta>0$, and the supporting set $\Omega$ has to be bounded
(so, its closure is a convex body). In this case, one may represent the 
density in the form $f = \varphi^{\beta}$, where $\varphi$ is an arbitrary 
positive concave function on $\Omega$, satisfying the normalization 
condition $\int_\Omega \varphi^{\beta}\,dx = 1$.

If $\k<0$, then $\tilde\k < 0$ and $f = V^{-\beta}$ 
(like in formula \eqref{eq:cvx-def}), where $V$ is an arbitrary positive convex function 
on $\Omega$, satisfying $\int_\Omega V^{-\beta}\,dx = 1$. Since $\beta=n-(1/\k)$
in this case, we must have $\beta>n$.

\vskip2mm
The following statement has been also well-known since the works of C. Borell,
cf. e.g. \cite[Theorem 4.5]{Bor75a}. (There it is assumed additionally that 
$0<\k',\k''<1/n$, while we will also need to consider the case when one of
$\k'$ or $\k''$ is negative. Nevertheless, Borell's result \cite[Theorem 4.2]{Bor75a} 
about $\kappa$-concavity of product measures covers the general case.)

\begin{proposition}\label{prop:cvx-conv}
Assume a probability measure $\mu$ is $\k'$-concave
on $\R^n$ and a probability measure $\nu$ is $\k''$-concave on $\R^n$. 
If $\k', \k'' \in [-1,1]$ satisfy
\be\label{eq:cvx-conv}
\k' + \k'' > 0, \qquad \frac{1}{\k} = \frac{1}{\k'} + \frac{1}{\k''},
\en
then their convolution $\mu * \nu$ is $\k$-concave.
\end{proposition}

Taking the limit $\k',\k'' \rightarrow 0$, one also obtains the 
log-concavity of the convolution of any two log-concave probability measures.

The argument is based on the following elementary property of the
$M_\k$-mean functions defined by
$$
M_\k^{(t)}(a,b) = (ta^\k + sb^\k)^{1/\k}, \qquad 
a,b \geq 0, \ \ 0<t<1, \ s=1-t,
$$
with the usual meaning in the cases $\k = -\infty$, $\k = +\infty$ and $\k = 0$,
as $\min\{a,b\}$, $\max\{a,b\}$ and $a^t b^s$, respectively.
(Note these functions appear on the right sides of \eqref{eq:kconc-def} and \eqref{eq:knconc-def}.)
Namely, under the condition \eqref{eq:cvx-conv}, for all real positive 
numbers $a',a'',b',b''$ and any $t \in (0,1)$,
$$
M_{\k'}^{(t)}(a',b')\, M_{\k''}^{(t)}(a'',b'')\, \geq \,
M_{\k}^{(t)}(a'a'',b'b'').
$$
Consequently, if $A = A' \otimes A''$ and $B = B' \otimes B''$ with
standard parallelotopes $A',B'$ in $\R^n$ of positive $\mu$-measure, 
and with standard parallelotopes $A'',B''$ in $\R^n$ of positive 
$\nu$-measure, then 
$$
tA + sB = (tA' + sB') \times (tA'' + sB''),
$$ 
and, using the definition \eqref{eq:kconc-def}, for the product measure 
$\lambda = \mu \otimes \nu$ we have:
\bee
\lambda(tA + sB)
 & = & 
\mu(tA' + sB')\, \nu(tA'' + sB'') \\
 & \geq &
M_{\k'}^{(t)}(\mu(A'),\mu(B'))\,
M_{\k''}^{(t)}(\nu(A''),\nu(B'')) \\
 & \geq & 
M_{\k}^{(t)}(\mu(A')\nu(A''),\mu(B')\nu(B'')) \\
 & = &
M_{\k}^{(t)}(\lambda(A),\,\lambda(B)).
\ene
That is, the Brunn-Minkowski-type inequality \eqref{eq:kconc-def} is fulfilled for the
measure $\lambda$ on $\R^{2n}$ in the class of all standard parallelotopes 
(of positive measure). By virtue of the standard bisection argument of 
Hadwiger-Ohmann \cite{HO56}, described, for example, in \cite{Bor74,Bor75a,BZ88:book}, one can 
extend \eqref{eq:kconc-def} from the class of standard parallelopipeds to arbitrary Borel 
sets $A$ and $B$, which means the $\k$-concavity of $\lambda$ on $\R^{2n}$. 
Finally, since $\mu * \nu$ represents the image of $\lambda$ under the 
linear map $(x,y) \rightarrow x+y$, the convolution is also $\k$-concave.

\vskip2mm
One particular case of Proposition~\ref{prop:cvx-conv} is the following well-known corollary:

\begin{corollary}\label{cor:convo}
If random vectors $X_1,\dots,X_m$ are independent
and uniformly distributed in convex bodies $A_1,\dots,A_m$ in $\R^n$, then
the sum
$$
X_1 + \dots + X_m
$$
has a $\frac{1}{mn}$--concave distribution supported on the convex body
$A_1 + \dots + A_m$.
\end{corollary}


\section{Entropy and volume of the support}
\label{sec:ent-vol}
\setcounter{equation}{0}

In this Section we bound the entropy of a $\k$-concave probability measure
on $\R^n$ with a positive parameter of convexity $\k$ in terms of the 
volume of its supporting set. 
Note that, for any random vector $X$ with values in $A$, there is a general 
upper bound
\be\label{eq:3ub}
h(X) \leq \log |A|.
\en
So our concern is how to estimate the entropy from below to get
\be\label{eq:3lb}
h(X) \geq -Cn + \log |A|
\en
with constants $C \geq 0$ depending only on the ``strength'' of convexity of
the density $f$ of $X$.

To proceed, we need some preparations. Given a measurable function $\varphi$ 
on a measurable set $A \subset \R^n$ and $p>0$, write
$$
\|\varphi\|_p = \bigg(\int_A |\varphi|^p\,dx\bigg)^{\!1/p}.
$$
The following Khinchin-type (or reverse H\"older) inequality for the class 
of concave functions is due to Berwald \cite{Ber47} (cf. \cite{Bor73a}).

\begin{lemma}\label{lem:berwald}
Given a concave function $\varphi \geq 0$ on a convex 
body $A$ in $\R^n$,
\be\label{eq:berwald}
\left(C_{n+q}^{n}\,|A|^{-1}\right)^{1/q}\ \|\varphi\|_q \leq
\left(C_{n+p}^{n}\,|A|^{-1}\right)^{1/p}\ \|\varphi\|_p, \qquad
0<p<q.
\en
\end{lemma}

Here and below we use the standard binomial coefficients
\be\label{eq:binom}
C_{q}^{n} = \frac{q(q-1)\dots(q-n+1)}{n!}.
\en
As easy to verify, the equality in \eqref{eq:berwald} is achieved for the linear function 
$f(x) = x_1 + \dots + x_n$ on the convex body 
\be\label{eq:simplex}
A = \{x \in \R^n: x_i > 0, \ x_1 + \dots + x_n < 1\}.
\en
 
Berwald's inequality may equivalently be stated for the class of ${\tilde\k}$-concave 
probability density functions $f$ on $A$ with ${\tilde\k}>0$, since then 
$f = \varphi^{1/{\tilde\k}}$ with concave $\varphi$. 
Inserting $\varphi = f^{\tilde\k}$ into \eqref{eq:berwald}, we get
$$
\left(C_{n+q}^{n}\,|A|^{-1}\right)^{1/q}\ \|f\|_{q{\tilde\k}}^{\tilde\k} \leq
\left(C_{n+p}^{n}\,|A|^{-1}\right)^{1/p}\ \|f\|_{p{\tilde\k}}^{\tilde\k}.
$$
Choose $p =\beta= 1/{\tilde\k}$ so that $\|f\|_{p{\tilde\k}} = \|f\|_1 = 1$. The inequality
is simplified (but does not lose generality):
$$
\left(C_{n+q}^{n}\,|A|^{-1}\right)^{1/q}\ \|f\|_{q{\tilde\k}}^{\tilde\k} \leq
\left(C_{n+1/{\tilde\k}}^{n}\,|A|^{-1}\right)^{{\tilde\k}}.
$$
Raising to the power $q$ and then substituting $q{\tilde\k}$ with $q$, we obtain
another equivalent form
$$
C_{n+q\beta}^{n}\,|A|^{-1}\, \int_A f(x)^q\,dx \ \leq \
\left(C_{n+\beta}^{n}\,|A|^{-1}\right)^{q},
$$
which holds true for any $q>1$. There is equality at $q=1$, so one may compare
the derivatives. First let us take logarithms of both the sides:
\be\label{eq:3inter1}
\log C_{n+q\beta}^{n} + \log\, |A|^{-1} + \log \int_A f(x)^q\,dx 
\ \leq \
q\, \log \left(C_{n+\beta}^{n}\,|A|^{-1}\right).
\en
By the definition \eqref{eq:binom},
$$
\frac{d}{dr}\, \log C_{n+r}^n = \sum_{i=1}^n \frac{1}{r+i}.
$$
Hence, differentiating \eqref{eq:3inter1} at $q=1$, we get
$$
\sum_{i=1}^n \frac{1}{1 + i/\beta} + \int_A f(x) \log f(x)\,dx 
\ \leq \
\log \left(C_{n+\beta}^{n}\,|A|^{-1}\right),
$$
or equivalently
\be\label{eq:3inter2}
h(X) \geq \log |A| + \sum_{i=1}^n \frac{1}{1 + i/\beta} - \log C_{n+\beta}^{n},
\en
assuming that $X$ has density $f$. 

Now, let us rewrite \eqref{eq:3inter2} in terms of the convexity parameter of the 
distribution of $X$ by applying the Borell characterization given in
Proposition~\ref{prop:cvx-meas}. Recall that if $X$ has an absolutely continuous $\k$-concave 
distribution supported on $A$ with $0 < \k \leq 1/n$, then it has a
${\tilde\k}$-concave density $f$, where ${\tilde\k} = \frac{\k}{1 - \k n}$.

\begin{proposition}\label{prop:cvx-ent}
Let $X$ be a random vector in $\R^n$ having an
absolutely continuous $\k$-concave distribution supported on a convex
body $A$ with $0 < \k \leq 1/n$. Then
\be\label{eq:3cvx-ent}
h(X) \geq \log |A| + \sum_{i=1}^n \frac{1}{1 + {\tilde\k} i} - \log C_{1/\k}^{n},
\en
where ${\tilde\k} = \frac{\k}{1 - \k n}$.
\end{proposition}

For each $\k$, equality in \eqref{eq:3cvx-ent} is attained for a special distribution
supported on the set $A$ defined in \eqref{eq:simplex}, with density $f(x)$ proportional 
to $(x_1 + \dots + x_n)^{1/{\tilde\k}}$. 
For example, if $\k = 1/n$, then
${\tilde\k} = +\infty$, and $X$ is to be uniformly distributed in $A$.
In this case, \eqref{eq:3cvx-ent} becomes just $h(X) \geq \log |A|$.

To simplify the bound \eqref{eq:3cvx-ent}, using again the notation $\beta = 1/{\tilde\k}$, 
we need to estimate from above the quantity
\be\label{eq:3inter3}
\log C_{n+\beta}^{n} - \sum_{i=1}^n \frac{\beta}{\beta + i} =
\sum_{i=1}^n \bigg[ \log \frac{\beta+i}{i} - \frac{\beta}{\beta + i}\bigg].
\en
In terms of $t = \beta/i$,
the general term in the sum on the right side may be written as
$$
\log \frac{\beta+i}{i} - \frac{\beta}{\beta + i} = \log(1+t) - \frac{t}{1+t},
$$
which is increasing in $t \geq 0$. Hence, the function
$s \rightarrow \log \frac{\beta+s}{s} - \frac{\beta}{\beta + s}$ is non-increasing. 
For any non-increasing continuous function $u = u(s) \geq 0$ in $s \geq 1$, 
one may use a general elementary bound
$$
\sum_{i=1}^n u(i) \leq u(1) + \int_1^n u(s)\,ds.
$$
In case of $u(s) = \log \frac{\beta+s}{s} - \frac{\beta}{\beta + s}$, we then get that
the sum on the right side of \eqref{eq:3inter3} is bounded by
\bee
 &
\hskip-90mm
\big[\log(\beta+1) - \frac{\beta}{\beta + 1}\big] + \int_1^n
\big[\log \frac{\beta+s}{s} - \frac{\beta}{\beta + s}\big]\,ds \\
\hskip40mm
 & = \
n \log(\beta+n) - n \log n - \frac{\beta}{\beta+1} \ \leq \
n \log \frac{\beta+n}{n}.
\ene

\noindent
Now, since $\beta = \frac{1}{{\tilde\k}} = \frac{1}{\k} - n$, we have $\beta+n = \frac{1}{\k}$
and therefore arrive at:

\begin{corollary}\label{cor:cvx-ent2}
Let $X$ be a random vector in $\R^n$ having an
absolutely continuous $\k$-concave distribution supported on a convex
body $A$ with $0 < \k \leq 1/n$. Then
$$
h(X) \geq \log |A| + n \log(\k n).
$$
\end{corollary}

Note when $\k = 1/n$, this bound is still sharp.

Now we can combine Corollaries~\ref{cor:convo} and \ref{cor:cvx-ent2} to obtain immediately:

\begin{proposition}\label{prop:vol-ent}
If random vectors $X_1,\dots,X_m$ are independent
and uniformly distributed in convex bodies $A_1,\dots,A_m$ in $\R^n$, then
their sum $S_m = X_1 + \dots + X_m$ has entropy, satisfying
$$
\log |A_1 + \dots + A_m| - n \log m\, \leq\,
h(S_m)\, \leq\ \log |A_1 + \dots + A_m|.
$$
\end{proposition}

Or, equivalently,
$$
\log \bigg|\frac{A_1 + \dots + A_m}{m}\bigg|\, \leq\,
h(S_m)\, \leq\ \log |A_1 + \dots + A_m|.
$$

In particular, for independent random vectors $X$ and $Y$ in $\R^n$ uniformly
distributed in convex bodies $A$ and $B$, respectively, we always have
$$
\log \bigg|\frac{A + B}{2}\bigg| \leq\, h(X+Y) \leq\ \log |A + B|.
$$
These are exactly the inequalities in \eqref{eq:volsum}, announced in the introductory 
section.


\section{Entropy and maximum of density}
\label{sec:ent-max}
\setcounter{equation}{0}

Any convex probability measure has a bounded density,
i.e., the $L^\infty$-norm $\|f\| = \sup_x f(x)$ of the density $f$
is finite (cf. \cite{Bob07}).
For sufficiently convex probability measures, the entropy may be related 
to  $\|f\|$  via the following proposition, proved in \cite{BM11:it}.

\begin{proposition}\label{prop:cvx-maxnorm}
Fix $\beta_0 > 1$. Assume a random vector $X$ in $\R^n$ has 
a density $f = V^{-\beta}$, where $V$ is a positive convex function
on the supporting set.
If $\beta \geq n+1$ and $\beta \geq \beta_0 n$, then
$$
\log\, \|f\|^{-1/n} \leq\,
\frac{1}{n}\, h(X)  \leq\, C_{\beta_0} + \log\, \|f\|^{-1/n}
$$
with some constant $C_{\beta_0}$ depending only on $\beta_0$.
\end{proposition}

The left inequality is general: It trivially holds without any convexity
assumption. The right inequality is  an asymptotic version of a result from \cite{BM11:it}
about extremal role of the multidimensional Pareto distributions.

Let us mention three immediate consequences of Proposition~\ref{prop:cvx-maxnorm}.
The first is the specialization to log-concave measures.


\begin{corollary}\label{cor:lc-maxnorm}
If a random vector $X$ in $\R^n$ has 
an absolutely continuous log-concave distribution with density $f$, then
$$
\log\, \|f\|^{-1/n} \leq\, \frac{1}{n}\, h(X)  \leq\, 
1 + \log\, \|f\|^{-1/n}.
$$
\end{corollary}

The right inequality is attained for the $n$-dimensional 
exponential distribution (with any parameter $\lambda>0$). This measure is 
concentrated on the positive orthant and has there density 
$f(x) = \lambda^n e^{-\lambda (x_1 + \dots + x_n)}$, $x_i > 0$.

%
Corollary~\ref{cor:lc-maxnorm} was observed by the first-named author around the year 2000 (motivated by relating the maximum 
of the density to the subgaussian norm). This observation was discussed with a few scholars but not published
and consequently was not widely known.
Independently, K.~Ball  observed this connection between $\|f\|$ and the entropy of $f$ 
for centrally symmetric, log-concave densities,
and publicized it in various lectures in 2003--06. He also proposed a program for 
approaching the hyperplane conjecture using this connection.  
Corollary~\ref{cor:lc-maxnorm} seems to have become well known
(to experts) soon after-- for example, it is implicit in the last part of the proof of
Theorem 7 of Fradelizi and Meyer \cite{FM08:1}, who showed the non-symmetric extension
using work of Fradelizi \cite{Fra97}. 
Unaware of parts of this history,  the authors in \cite{BM11:it} first explicitly wrote down Corollary~\ref{cor:lc-maxnorm} 
in the form given above.

It was observed in \cite{BM11:it} that Corollary~\ref{cor:lc-maxnorm} 
can be written as a Gaussian comparison inequality.
Specifically, for any log-concave density $f$, we have
\be\label{eq:bm11-d}
-\frac{1}{2} \,\leq \, \frac{1}{n}\, h(Z) - \frac{1}{n}\, h(X) \, 
\leq \,\frac{1}{2},
\en
where $Z$ is any Gaussian random vector in $\R^n$ with the same maximal 
value of the density as $f$. On the other hand, if we replace the assumption
about the maximum with the requirement that $Z$ has the same covariance 
matrix as $X$, one may consider a different inequality of a similar form
$$
0 \,\leq \, \frac{1}{n}\, h(Z) - \frac{1}{n}\, h(X) \, 
\leq \,C.
$$
Whether or not it is possible to choose here an absolute constant $C$
(to serve the class of all log-concave densities) represents a question
equivalent to the hyperplane conjecture (cf. \cite{BM11:it} for discussion, although
the idea of such an equivalence should be credited to K.~Ball as mentioned above).
Let us also note that the dimension-free Gaussian comparison inequality \eqref{eq:bm11-d}
is similar in spirit to the main result of Section~\ref{sec:ent-vol}. 
Specifically, if for $\k>0$, $f$ is a density of a $\k$-concave random vector $X$ taking values in the 
convex body $A$, and if $\text{U}_A$ is the uniform distribution on $A$,
\eqref{eq:3ub}--\eqref{eq:3lb} are equivalent to the statement
$$
0 \,\leq\, 
\frac{1}{n}\,  h(\text{U}_A) - \frac{1}{n}\, h(X) \,\leq \, C.
$$


We proceed to describe two further consequences of Proposition~\ref{prop:cvx-maxnorm}.

\begin{corollary}\label{cor:innerprod}
If random vectors $X$ and $Y$ in $\R^n$ are
independent and have symmetric log-concave densities $f$ and $g$, respectively, 
then
$$
\bigg(\int f(x) g(x)\,dx\bigg)^{\!-2/n} \leq \, H(X+Y) \, \leq \,
e^2 \bigg(\int f(x) g(x)\,dx\bigg)^{\!-2/n}.
$$
\end{corollary}

Note that, by the symmetry assumption, the convolution
$f*g(x) = \int f(x-y)g(y)\,dy$ represents a symmetric log-concave density.
Hence, it attains maximum at the origin, so that
$$
\|f*g\| = f*g(0) = \int f(x) g(x)\,dx.
$$

Now, returning to the convex body case, let us combine Proposition~\ref{prop:vol-ent} 
with Corollary~\ref{cor:lc-maxnorm} applied to $X=S_m$.

\begin{corollary}\label{cor:vol-maxnorm}
Let $X_1,\dots,X_m$ be independent and uniformly 
distributed in convex bodies $A_1,\dots,A_m$ in $\R^n$, and let $f_m$ be
the density of the sum $S_m = X_1 + \dots + X_m$. Then
\be\label{eq:vol-maxnorm}
1\leq \|f_m\| \cdot |A_1 + \dots + A_m| \leq (me)^n.
\en
\end{corollary}

To illustrate possible implications, again assume we have two 
convex bodies $A$ and $B$ in $\R^n$, and let $X,Y$ be independent and 
uniformly distributed in $A$ and $-B$, respectively, that is, with densities
$f(x) = \frac{1}{|A|}\, 1_A(x)$, $g(x) = \frac{1}{|B|}\, 1_B(-x)$.
Their convolution
$$
f * g(x) = \frac{1}{|A|\, |B|} \int 1_A(x-y)\, 1_B(-y)\,dy =
\frac{|(A-x) \cap B|}{|A|\, |B|}
$$ 
is supported on $\Omega = A-B$, and \eqref{eq:vol-maxnorm} yields
$$
\sup_{x}
|(A-x) \cap B| \cdot |A-B| \leq\, (2e)^n\, |A|\,|B|.
$$
In fact, by a more careful application of Berwald's inequality (see \cite{BM10:norm} for details), 
the constant here may be slightly improved to get
\be\label{eq:diffbody}
\sup_{x} |(A-x) \cap B| \cdot |A-B| \leq C_{2n}^n\, |A|\,|B|.
\en
This inequality is known as the  Rogers-Shephard inequality \cite[Equation 14]{RS58:1}.
When $A=B$, and taking $x=0$, it yields the Rogers-Shephard difference 
body inequality $|A-A| \leq C_{2n}^n\, |A|$, with the sharp dimensional
constant \cite{RS57}. 

Note also that, since $C_{2n}^n < 4^n$, both the sides of \eqref{eq:diffbody} 
are of a similar order in the sense that
\be\label{eq:diffbody2}
|A|^{1/n}\,|B|^{1/n}\, \leq\, \sup_{x} |(A-x) \cap B|^{1/n}\, |A-B|^{1/n} \,
\leq\, 4\, |A|^{1/n}\,|B|^{1/n}.
\en
Here the left inequality is just the bound
$\|f*g\| \geq |\Omega|^{-1} \int f*g(x)\,dx = |A-B|^{-1}$.

In particular, for all symmetric convex bodies $A$ and $B$ in $\R^n$,
\be\label{eq:symmsum}
|A|^{1/n} \ |B|^{1/n} \leq |A \cap B|^{1/n}\, |A+B|^{1/n} \leq
4\,|A|^{1/n} \ |B|^{1/n}.
\en


\section{Essential support of convex measures}
\label{sec:supp}
\setcounter{equation}{0}

Although log-concave and more general convex measures on $\R^n$ do not have
bounded supports, it is important to find a suitable form of Proposition~\ref{prop:cvx-ent}
and its Corollary~\ref{cor:cvx-ent2} which give bounds on the entropy for compactly
supported convex measures. As it turns out, an ``essential'' part of
any convex measure is supported on a certain convex body, and moreover its 
volume may be related to the entropy of the measure. For the class of 
log-concave probability measures an observation of this concentration type 
was first made by B. Klartag and V. D. Milman in \cite{KM05}, who proved the 
following statement (cf. \cite[Corollary 2.4]{KM05} or \cite[Corollary 5.1]{Kla07:1}).

\begin{proposition}\label{prop:lc-aep}
For any log-concave probability measure
$\mu$ on $\R^n$ with density $f$,
$$
\mu\big\{f \geq c_0^n\, \|f\|\big\} \geq 1-c_1^n
$$
with some universal constants $c_0,c_1 \in (0,1)$.
\end{proposition}

In fact, at the expense of $c_0$ one may choose $c_1$ to be as small as we 
wish. See also \cite{BM11:aop} for refinements.

Our next step is to prove the following analogue of Proposition~\ref{prop:lc-aep}
for the class of convex measures.

\begin{proposition}\label{prop:cvx-aep}
Let $\mu$ be a probability measure on $\R^n$ with 
density $f = V^{-\beta}$, where $V$ is a convex function
on the supporting set. If $\beta \geq n+1$ and $\beta \geq \beta_0 n$ 
with $\beta_0 > 1$, then
\be\label{eq:cvx-aep}
\mu\big\{f \geq c_0^n\,\|f\|\big\} \geq \frac{1}{2},
\en
for some $c_0 \in (0,1)$ depending on $\beta_0$, only.
\end{proposition}

At the expense of the constant $c_0$ the bound $1/2$ on the right side 
of \eqref{eq:cvx-aep} can be replaced with any prescribed number $p \in (0,1)$.
The convex body 
$$
K_f = \big\{x \in \R^n: f(x) \geq c_0^n\,\|f\|\big\}
$$
may be viewed as the ``$\frac{1}{2}$--support'' or ``essential support'' of the 
measure~$\mu$. (The latter interpretation can be better justified by taking  $p$ to be some 
fixed number that is close to 1, but this is not needed for our purposes.)

\begin{proof}
By the Borell characterization theorem (Proposition~\ref{prop:cvx-meas}),
$\mu$ is supported on an open convex set $\Omega$, where $V$ is positive
and convex. Without loss of generality, assume $V$ attains minimum at some
point $x_0 \in \Omega$, and moreover
$$
V(x_0) = \min_{x \in \Omega} V(x) = 1,
$$
which corresponds to $\|f\| = 1$. Introduce sublevel convex sets
$$
A(\lambda) = \{x \in \Omega: f(x) > \lambda\}, \qquad 0 < \lambda < 1,
$$
and similarly
$$
A'(t) = \{x \in \Omega: V(x) < 1+t\}, \qquad t>0.
$$
Thus $A(\lambda) = A'(\lambda^{-1/\beta}-1)$. By the Brunn-Minkowski inequality,
the function $\varphi(t) = |A'(t)|$ is $\frac{1}{n}$--concave in $t>0$,
that is, $\varphi(t) = \psi(t)^n$ for some concave function $\psi$, which
is also non-negative and non-decreasing. We may assume that
$\varphi(0+) = 0$ and similarly for $\psi$. Integrating by parts, we have
$$
\int V(x)^{-\beta}\,dx = \int_0^{+\infty} (1+t)^{-\beta} d\varphi(t)
 = \beta \int_0^{+\infty} (1+t)^{-\beta-1} \varphi(t)\,dt,
$$
that is,
\be\label{eq:5a}
\beta \int_0^{+\infty} (1+t)^{-\beta-1}\,\psi(t)^n\, dt = 1.
\en
Fix $t_0 > 0$ and write similarly
\bee
1 - \mu(A'(t_0)) \ = \
\int_{\{V \geq 1+t_0\}} V(x)^{-\beta}\,dx
                 & = &
\int_{t_0}^{+\infty} (1+t)^{-\beta} d\varphi(t) \\
                 & = &
\beta \int_{t_0}^{+\infty} (1+t)^{-\beta-1}\,\varphi(t)\,dt -
(1+t_0)^{-\beta} \varphi(t_0),
\ene
so,
\be\label{eq:5b}
1 - \mu(A'(t_0))\, \leq\, 
\beta \int_{t_0}^{+\infty} (1+t)^{-\beta-1}\,\psi(t)^n\,dt.
\en

Now, we need to estimate from above the integral \eqref{eq:5b} subject to \eqref{eq:5a}. 
By concavity and monotonicity of $\psi$,
$$
\psi(t) \geq
\begin{cases}
ct, & \text{for} \ \ 0 < t < t_0 \\
ct_0,  & \text{for} \ \ t \geq t_0
\end{cases}
$$
where $c = \psi(t_0)/t_0$. Hence, integrating just over the interval 
$(0,t_0)$, we get
$$
\int_0^{+\infty} (1+t)^{-\beta-1}\,\psi(t)^n\,dt\, \geq\, 
c^n \int_0^{t_0} \frac{t^n}{(1+t)^{\beta + 1}}\, dt\, =\,
c^n \int_{s_0}^1 s^{\beta - n - 1} (1-s)^n\,ds,
$$
where $s_0 = 1/(1+t_0)$ and where we used the substitution $s = 1/(1+t)$.
Hence, by \eqref{eq:5a},
\be\label{eq:5c}
c^n\, \leq \,
\frac{1}{\beta \int_{s_0}^1 s^{\beta - n - 1} (1-s)^n\,ds}.
\en
On the other hand, using $\psi(t) \leq ct$, which holds for all $t>t_0$,
we obtain that
$$
\int_{t_0}^{+\infty} (1+t)^{-\beta-1}\,\psi(t)^n\,dt\, \leq\, 
c^n \int_{t_0}^{+\infty} \frac{t^n}{(1+t)^{\beta + 1}}\, dt\, =\,
c^n \int_0^{s_0} s^{\beta - n - 1} (1-s)^n\,ds.
$$
Combining \eqref{eq:5b} and \eqref{eq:5c}, we get
\be\label{eq:5d}
1 - \mu(A'(t_0))\, \leq\ \frac{\P\{\xi < s_0\}}{\P\{\xi > s_0\}}, \qquad
s_0 = \frac{1}{1+ t_0},
\en
where $\xi$ is a random variable having the beta-distribution
with parameters $(\beta - n,n+1)$, that is, with density
$$
p(s)\, =\, \frac{1}{B(\beta - n,n+1)}\ s^{\beta - n - 1} (1-s)^n, \qquad 
0 < s < 1.
$$

Now, to better understand the expression in \eqref{eq:5d}, it is useful
to relate the beta distribution to the gamma distribution. It is
a well-known fact in probability that in the sense of distributions
$$
\xi = \frac{\Gamma_{\beta - n}}{\Gamma_{\beta - n} + \Gamma_{n+1}},
$$
where $\Gamma_{\beta - n}$ and $\Gamma_{n+1}$ are independent random variables, 
having the gamma distribution with shape parameters $\beta - n$ and $n+1$ 
respectively (and with the scale parameter~1). In particular, one may write
$\Gamma_{n+1} = \zeta_1 + \dots + \zeta_{n+1}$, where the $\zeta_i$'s are
independent and have a standard exponential distribution.

Note that the inequality 
$\xi < s_0$ is solved as $\Gamma_{n+1} > t_0\, \Gamma_{\beta - n}$. 
Consequently, \eqref{eq:5d} takes the form
\be\label{eq:5e}
1 - \mu(A'(t_0))\, \leq\, 
\frac{\P\{\Gamma_{n+1} > t_0\,
\Gamma_{\beta - n}\}}{\P\{\Gamma_{n+1} < t_0\,\Gamma_{\beta - n}\}}.
\en

Using Chebyshev's inequality, for any $\alpha > 1$ and $s \in (0,1)$, 
and actually with optimal $s = 1 - 1/\alpha$, one may write
$$
\P\{\Gamma_{n+1} > \alpha (n+1)\} \leq 
(\E e^{s\zeta_1})^{n+1}\, e^{-\alpha s (n+1)} =
\bigg(\frac{e^{-\alpha s}}{1-s}\bigg)^{n+1} =
\left(e \cdot \alpha e^{-\alpha}\right)^{n+1}.
$$
Take, for example, $\alpha = 4$, in which case the above gives
\be\label{eq:5f}
\P\{\Gamma_{n+1} > 4 (n+1)\}\, \leq\, \bigg(\frac{4}{e^3}\bigg)^{n+1} < \,
\bigg(\frac{1}{5}\bigg)^{n+1}.
\en
Hence,
\bee
\P\{\Gamma_{n+1} > t_0\,\Gamma_{\beta - n}\} & = & 
\P\{\Gamma_{n+1} > t_0\,\Gamma_{\beta - n}, \ \Gamma_{n+1} > 4 (n+1)\} \\
 & & \hskip-4mm + \,
\P\{\Gamma_{n+1} > t_0\,\Gamma_{\beta - n}, \ \Gamma_{n+1} < 4 (n+1)\} \\ 
               & < &
5^{-(n+1)} + \P\big\{\Gamma_{\beta - n} < \frac{4 (n+1)}{t_0}\big\}.
\ene
In terms of $t_0 = \frac{4 (n+1)}{T}$, where $T > 0$ will be choosen later on,
we thus obtain that
\be\label{eq:5g}
\P\{\Gamma_{n+1} > t_0\,\Gamma_{\beta - n}\} < 
5^{-(n+1)} + \P\{\Gamma_{\beta - n} < T\}.
\en
Now,
\bee
\P\{\Gamma_{\beta - n} < T\} & = &
\frac{1}{\Gamma(\beta - n)} \int_0^T x^{\beta - n -1}\, e^{-x}\,dx \\
              & < &
\frac{1}{\Gamma(\beta - n)} \int_0^T x^{\beta - n -1}\,dx
              \ = \
\frac{T^{\beta - n}}{\Gamma(\beta - n + 1)} \ = \ \frac{T^\alpha}{\Gamma(\alpha + 1)},
\ene
where we put $\alpha = \beta - n$ (which is positive). Take
$T = \frac{1}{4}\, \E \Gamma_{\beta - n} = \frac{\alpha}{4}$, so that
\be\label{eq:5h}
\P\{\Gamma_{\beta - n} < T\}\ \leq\ 
\frac{(\frac{\alpha}{4})^\alpha}{\Gamma(\alpha + 1)}.
\en

We claim that the right side of \eqref{eq:5h} does not exceed $1/4$ for any
$\alpha \geq 1$. Here we use the following observation.
If $\zeta$ is a random variable with the standard exponential distribution,
then $\E \zeta^\alpha = \Gamma(\alpha + 1)$ and the claim takes the form
\be\label{eq:5i}
h(\alpha) \equiv \log \E \bigg(\frac{\zeta}{\alpha}\bigg)^\alpha \geq 
\log 4 -\alpha \log 4.
\en
But as shown in \cite{Bob03:gafa1}, the function $h$ is always concave on the positive
half-axis $\alpha>0$, whenever $\zeta>0$
has a log-concave distribution. Hence, it is enough to verify \eqref{eq:5i} for
$\alpha = 1$ and $\alpha = +\infty$.  In our particular case, at the left 
endpoint there is equality, while Stirling's formula shows that \eqref{eq:5i} also
holds at infinity.

Thus, $\P\{\Gamma_{\beta - n} < \frac{\beta - n}{4}\} \leq \frac{1}{4}$ 
whenever $\beta -n \geq 1$, and for
$t_0 = \frac{4 (n+1)}{T} = \frac{16 (n+1)}{\beta - n}$ the inequality 
\eqref{eq:5g} yields
$$
\P\{\Gamma_{n+1} > t_0\,\Gamma_{\beta - n}\} < 
5^{-(n+1)} + \frac{1}{4} < \frac{1}{3},
$$
so that by \eqref{eq:5e},
\be\label{eq:5j}
1 - \mu\left(A'(t_0)\right)\, \leq\ 
\frac{\frac{1}{3}}{1 - \frac{1}{3}} = \frac{1}{2}.
\en

Finally, recall that $A(\lambda) = A'(\lambda^{-1/\beta}-1)$ or
$$
A'(t_0) = A(\lambda) \quad {\rm with} \quad \lambda =
\bigg(1 + 16\, \frac{n+1}{\beta - n}\bigg)^{-\beta}.
$$
By \eqref{eq:5j}, for this value we have 
$\mu(A(\lambda)) = \mu\{f > \lambda\} \geq \frac{1}{2}$.
We need an estimate of the form $\lambda \geq c^n$, with some $c > 0$
depending on $\beta$. The latter is equivalent to
\be\label{eq:5k}
\beta \log\bigg(1 + 16\, \frac{n+1}{\beta - n}\bigg) \leq n \log C \quad
(C = 1/c)
\en
which is indeed fulfilled in the range $\beta \geq \beta_0 n$ with
$\beta_0 > 1$ and $C = C(\beta_0)$. However, it is not true
for $\beta = n + O(1)$. 

To simplify \eqref{eq:5k}, one may just use the elementary bound $\log (1+x) \leq x$, 
so that \eqref{eq:5k} would follow from
$$
16 \beta\, \frac{n+1}{\beta - n} \leq n \log C 
$$
which holds for all $\beta \geq \beta_0 n$ with 
$C = \exp\{32 \beta_0/(\beta_0 - 1)\}$. Thus, Proposition~\ref{prop:cvx-aep} is proved with
\be\label{eq:5l}
c_0 = \exp\{-32 \beta_0/(\beta_0 - 1)\}.
\en
\end{proof}

\vskip2mm
\begin{remark}\label{rmk:lc-aep}
A slight modification of the above argument leads to 
Proposition~\ref{prop:lc-aep}. Indeed, let $f$ be a log-concave density such that $\|f\|=1$.
Write once more the inequality \eqref{eq:5e} with $t_0 = t/\beta$, $t>0$, and recall 
the relation $A(\lambda) = A'(\lambda^{-1/\beta}-1)$. Hence, \eqref{eq:5e} takes the form
$$
1 - \mu\left(A\left(\big(1 + t/\beta\big)^{-\beta}\,\right)\right)\, 
\leq\, \frac{\P\{\Gamma_{n+1} > 
t\,\Gamma_{\beta - n}/\beta\}}{\P\{\Gamma_{n+1} < t\,\Gamma_{\beta - n}/\beta\}}.
$$
Letting $\beta \rightarrow +\infty$ and using 
$\Gamma_{\beta - n}/\beta \rightarrow 1$ in probability (according to the weak 
law of large numbers), we arrive in the limit at
$$
1 - \mu\left(A(e^{-t})\right)\, \leq\, 
\frac{\P\{\Gamma_{n+1} > t\}}{\P\{\Gamma_{n+1} < t\}}, 
\qquad t>0.
$$
Choose, for example, $t = 8n \geq 4(n+1)$. Then, by \eqref{eq:5f},
$$
1 - \mu\left(A(e^{-8n})\right)\, <\, \frac{5^{-(n+1)}}{1 - 5^{-(n+1)}} < 
\frac{1}{5^n}.
$$
Hence, Proposition~\ref{prop:lc-aep} holds with $c_0 = e^{-8}$ and $c_1 = 1/5$.
\end{remark}

\begin{remark}\label{rmk:finmeas}
By homogeneity, Propositions~\ref{prop:lc-aep}--\ref{prop:cvx-aep} may be stated for
finite convex measures. In particular, if $f = V^{-\beta}$ is Lebesgue 
integrable, where $V$ is a positive convex function, and $\beta \geq n+1$ 
and $\beta \geq \beta_0 n$ with $\beta_0 > 1$, then
\be\label{eq:5m}
\int f(x)\,dx \leq 2\, \int_{K_f} f(x)\,dx \leq 2\, \|f\|\,|K_f|.
\en
Recall that
$$
K_f = \left\{x \in \Omega: f(x) \geq c_0^n\, \|f\|\right\}
$$
is the essential support of $\mu$ with $c_0$ depending on $\beta_0$, only.
(One may choose the constant \eqref{eq:5l}).
\end{remark}

\vskip2mm
To illustrate how Proposition~\ref{prop:cvx-aep} may be applied, note that
$2\, \|f\|\,|K_f| \geq 1$, according to \eqref{eq:5m}. On the other hand,
since $1 \geq \int_{K_f} f(x)\,dx \geq c_0^n \|f\|\, |K_f|$, 
we have that $\|f\|\,|K_f| \leq c_0^{-n}$. Thus, 
\be\label{eq:5n}
\frac{1}{2}\,\|f\|^{-1/n} \leq \,|K_f|^{1/n} \leq c_0^{-1}\,\|f\|^{-1/n}.
\en 
But by Proposition~\ref{prop:cvx-maxnorm}
, if a random vector $X$ has distribution $\mu$,
$$
1 \leq H(X)\, \|f\|^{2/n} \leq C
$$
with constants $C$, depending on $\beta$, only (in case of the range as in
Proposition~\ref{prop:cvx-aep}). Hence, we arrive at:

\begin{corollary}\label{cor:typset}
Let a random vector $X$ in $\R^n$ have 
a density $f = V^{-\beta}$, where $V$ is a positive convex function
on the supporting set.
If $\beta \geq n+1$ and $\beta \geq \beta_0 n$ with $\beta_0 > 1$, then
$$
C'_{\beta_0}\, |K_f|^{2/n}\, \leq\, H(X)\, \leq\, C''_{\beta_0}\, |K_f|^{2/n},
$$
where $K_f$ is the essential support of the distribution of $X$, and where 
$C''_{\beta_0} > C'_{\beta_0} > 0$ depend on $\beta_0$, only.
\end{corollary}


\section{$M$-position for convex bodies and measures}
\label{sec:m-pos}
\setcounter{equation}{0}

The so-called $M$-position of convex bodies was introduced by V. D. Milman 
in connection with reverse forms of the Brunn-Minkowski inequality, cf. \cite{Mil86}.
By now  several equivalent definitions of this important
concept are known, and for our purposes we choose one of them. We refer an interested
reader to the subsequent works \cite{Mil88:1}, \cite{Mil88:2} and the book by G. Pisier \cite{Pis89:book}, 
which also contains historical remarks; cf. also \cite{Bob11} for the relationship
between $M$-position and isotropicity.

For any convex body $A$ in $\R^n$, define
$$
M(A) = \sup_{|{\cal E}| = |A|} \frac{|A \cap {\cal E}|^{1/n}}{|A|^{1/n}},
$$
where the supremum is over all ellipsoids ${\cal E}$ with volume
$|{\cal E}| = |A|$. The main result of V. D. Milman may be stated as follows: 

\begin{proposition}\label{prop:M}
If $A$ is a symmetric convex body in $\R^n$, 
then with some universal constant $c>0$
\be\label{eq:M-ell}
M(A) \geq c.
\en
\end{proposition}

By the Brunn-Minkowski and Rogers-Shephard difference body inequalities, 
for any convex body $A$ in $\R^n$, we have $M(A) \geq \frac{1}{2}\, M(A-A)$. 
Hence, the symmetry assumption in \eqref{eq:M-ell} may be removed.
(That this may be done was first noticed by V. Milman and A. Pajor in \cite{MP00},
using a different but equivalent definition of $M$-ellipsoids.)

If $|A \cap {\cal E}|^{1/n} \geq c\, |A|^{1/n}$ with a universal constant 
$c>0$, then ${\cal E}$ is called an $M$-ellipsoid, or Milman's ellipsoid. 
It can be shown with the help of the reverse Santalo inequality
due to Bourgain and Milman and using a bound such as \eqref{eq:symmsum} that, if
${\cal E}$ is a (symmetric) $M$-ellipsoid for a symmetric convex body $A$,
then the dual ellipsoid ${\cal E}^{\rm o}$ is an $M$-ellipsoid for the dual 
body $A^{\rm o}$ (although with different absolute constants).

It follows from the definition that, for any convex body $A$ in 
$\R^n$, one can find an affine volume preserving map $u:\R^n \rightarrow \R^n$
such that $u(A)$ has a multiple of the unit centered Euclidean ball as an 
$M$-ellipsoid. In that case, one says that $u(A)$ is in $M$-position. 
Or equivalently, $A$ is in $M$-position, if
\be\label{eq:M-ell2}
|A \cap D|^{1/n} \geq c\, |A|^{1/n},
\en
where $D$ is a Euclidean ball with center at the origin, 
such that $|D| = |A|$, and where $c>0$ is universal.

The definition of an $M$-position may naturally be extended to the class 
of convex measures. Let $\mu$ be a convex probability measure 
on $\R^n$ with density $f$ such that $\|f\|=1$. Then we say 
that $\mu$ is in $M$-position (with constant $c>0$), if
\be\label{eq:M-cvx}
\mu(D)^{1/n} \geq c,
\en
where $D$ is a Euclidean ball with center at the origin of volume
$|D|=1$. Correspondingly, Proposition~\ref{prop:M} can be generalized to a class
of convex measures.

\begin{proposition}\label{prop:M-cvx}
Let $\mu$ be a probability measure on $\R^n$ 
with density $f = V^{-\beta}$ such that $\|f\| \geq 1$, where $V$ is a convex 
function on the supporting set. If $\beta \geq n+1$ and $\beta \geq \beta_0 n$ 
with $\beta_0 > 1$, then $\mu$ may be put in a position where
$$
\mu(D)^{1/n} \geq c_0
$$
for some $c_0 \in (0,1)$ depending on $\beta_0$
$($where $D$ is the Euclidean ball of volume one$)$.
\end{proposition}

By saying ``put'' we mean that, for some affine volume preserving map 
$u:\R^n \rightarrow \R^n$, the image $u(\mu) = \mu u^{-1}$ of the measure
$\mu$ under the map $u$ is in $M$-position.

In particular, any log-concave probability measure $\mu$ on $\R^n$ 
with density $f$ such that $\|f\|=1$ may be put in $M$-position
with a universal constant.

\begin{proof}
We may assume that $\|f\|=1$.
By Proposition~\ref{prop:cvx-aep}, for some constant $c_0 > 0$, 
which only depends on $\beta_0$, the essential support of $\mu$, i.e.,
the set $K_f = \{f(x) \geq c_0^n\}$ has measure 
$\mu(K_f) \geq 1/2$. Hence, as was already noted in Remark~\ref{rmk:finmeas}, we have
$$
\frac{1}{2} \leq \,|K_f|^{1/n} \leq c_0^{-1}.
$$

Put $K_f' = \frac{1}{|K_f|^{1/n}}\, K_f$, which is a convex body 
with volume $|K_f'|=1$. 

One may assume that $K_f'$ contains the origin and is already in $M$-position (otherwise, apply to $K_f'$ a linear, volume preserving map $u$ to put it 
in $M$-position and 
consider the image $u(\mu)$ in place of $\mu$). 
We claim that if $K_f'$ is in $M$-position, then $\mu$ is also in $M$-position.

Indeed, if $D$ is the Euclidean ball with center at the origin of volume
$|D|=1$, then \eqref{eq:M-ell2} is satisfied for the set $A = K_f'$ with a universal
constant $c>0$.
Since $K_f' \subset 2 K_f$, we have
$|K_f' \cap D| \leq |2K_f \cap D| \leq 2^n |K_f \cap D|$. Therefore,
$$
\mu(D) \geq \int_{K_f \cap D} f(x)\,dx \geq c_0^n\, |K_f \cap D| \geq
c_0^n \cdot 2^{-n} |K_f' \cap D| \geq \bigg(\frac{c_0 c}{2}\bigg)^n.
$$
Proposition~\ref{prop:M-cvx} is proved.
\end{proof}


\section{Submodularity of entropy and implications}
\label{sec:submod}
\setcounter{equation}{0}

In the proof of Theorem~\ref{thm:repi} we apply a general submodularity property of the entropy 
functional, recently obtained in \cite{Mad08:itw}. We state it below in the particular 
case of three random vectors.

\begin{proposition}\label{prop:submod}
Given independent random vectors $X$, $Y$, $Z$
in $\R^n$ with absolutely continuous distributions, we have
$$
h(X+Y+Z) + h(Z) \leq h(X+Z) + h(Y+Z)
$$
provided that all entropies are well-defined.
\end{proposition}

In particular, let $X,Y,Z$ be uniformly distributed in arbitrary convex bodies
$A,B,D$, respectively. By Proposition~\ref{prop:vol-ent} with $m=3$, we then obtain that
$$
|A + B + D|^{1/n}\, |D|^{1/n} \leq\, 3\, |A + D|^{1/n}\,|B + D|^{1/n}.
$$

Let us comment on the relationship between Proposition~\ref{prop:M} and the reverse 
Brunn-Minkowski inequality from our 
point of view. The fact that the former implies the latter
is contained in V.~Milman's original papers \cite{Mil86, Mil88:2, Mil88:1} (cf. Pisier \cite[Corollary 7.3]{Pis89:book}) 
and is based on arguments involving metric entropy rather than measure-theoretic entropy. 

\begin{corollary}
The existence of $M$-ellipsoids for symmetric, convex bodies is equivalent to the reverse Brunn-Minkowski inequality.
\end{corollary}

\begin{proof}
Using the monotonicity of entropy, i.e., $h(X+Y+Z) \geq h(X+Y)$, we also
have another variant with a somewhat better constant
\be\label{eq:vol-submod}
|A + B|^{1/n}\, |D|^{1/n} \leq\, 2\, |A + D|^{1/n}\,|B + D|^{1/n}.
\en
If, furthermore, all these convex bodies are symmetric and have volume one, by 
\eqref{eq:symmsum} applied to the couples $(A,D)$ and $(B,D)$, we get from \eqref{eq:vol-submod} that
\be\label{eq:vol-submod2}
\frac{1}{|A \cap B|^{1/n}}\, \leq\,
|A + B|^{1/n}\, \leq\, \frac{32}{|A \cap D|^{1/n}\,|B \cap D|^{1/n}}.
\en

Therefore, if $A$ and $B$ are in $M$-position and have volume one, and
$D$ is the Euclidean ball of volume one, the right inequality in \eqref{eq:vol-submod2} 
together with the definition \eqref{eq:M-ell2} of $M$-position leads to the 
reverse Brunn-Minkowski inequality in the form \eqref{eq:BM-norm} with an identity linear operator,
$$
|A + B|^{1/n}\, \leq\, C.
$$
Note that the symmetry assumption in this conclusion can be removed by 
applying the above to the sets $A' = \frac{1}{|A-A|^{1/n}}\, (A-A)$ and
$B' = \frac{1}{|B-B|^{1/n}}\, (B-B)$ and making use of the Rogers-Shephard 
difference body inequality.

The converse statement that the reverse Brunn-Minkowski inequality implies
Proposition~\ref{prop:M} can be based on the left side of \eqref{eq:vol-submod2}. Indeed, let
$A$ be a symmetric convex body in $\R^n$ with volume one.
Our hypothesis includes, in particular, that for some linear volume preserving 
map $u:\R^n \rightarrow \R^n$, the set $\widetilde A = u(A)$ satisfies
$$
|\widetilde A + D|^{1/n}\, \leq\, C,
$$
where $D$ is the Euclidean ball of volume one, as before.
But then the left inequality in \eqref{eq:vol-submod2} being written for the couple 
$(\widetilde A,D)$ indicates that $\widetilde A$ is in $M$-position 
with constant $c = 1/C$.
\end{proof}

The following property of convex bodies in $M$-position
is well-known. (It can be obtained, for instance, by comparing the left and 
right sides of inequality \eqref{eq:vol-submod2}). If $A$ and $B$ are symmetric convex 
bodies in $M$-position of volume one, then
$$
|A \cap B|^{1/n} \geq c_1,
$$
where $c_1 = c^2/32$ and $c$ is Milman's constant in \eqref{eq:M-ell}. If we drop the 
volume assumption, the above may be applied to the sets 
$\frac{1}{|A|^{1/n}}\, A$ and $\frac{1}{|B|^{1/n}}\, B$, which leads to the following corollary.

\begin{corollary}\label{cor:M-sum}
Let $A$ and $B$ be symmetric convex bodies in $\R^n$ that are in $M$-position.
Then
$$
|A \cap B|^{1/n}\, \geq\, c_1\, \min\{|A|^{1/n}, |B|^{1/n}\}.
$$
\end{corollary}

Without the symmetry assumption, we still have a similar property
\be\label{eq:M-sum2}
\sup_x |(A-x) \cap B|^{1/n}\, \geq\, c_1\, \min\{|A|^{1/n}, |B|^{1/n}\}.
\en
Indeed, by \eqref{eq:diffbody2} and \eqref{eq:vol-submod}, the inequality \eqref{eq:vol-submod2} may be generalized as
$$
\frac{1}{\sup_x |(A-x) \cap B|^{1/n}}\ \leq\,
|A - B|^{1/n}\, \leq\, \frac{32}{|A \cap D|^{1/n}\,|B \cap D|^{1/n}},
$$
where $A,B,D$ are convex bodies of volume one and such that $D$ is symmetric.

It was mentioned in Section~\ref{sec:intro} that we provide a technology for 
going from entropy to volume estimates. Let us illustrate this in the context of
the submodularity phenomenon discussed here. Indeed, as described in
\cite[Theorem III]{Mad08:itw}, one consequence of submodularity is the following inequality.

\begin{lemma}\label{lem:fracsub}
Let $X$ and $Y_{1},\ldots, Y_{m}$ be independent $\R^{n}$-valued random vectors with finite entropies.
Let $\collS_k$ denote the collection of all subsets of $[m]=\{1,\ldots, m\}$ that are of cardinality $k$.
Then
$$
h\bigg(X+ \sum_{i\in[m]} Y_i \bigg) - h(X) \leq 
\frac{1}{\binom{m-1}{k-1}} \sum_{\setS\in\collS_k}  \bigg[ h\bigg( X+\sum_{i\in\setS} Y_{i}\bigg) - h(X) \bigg].
$$
\end{lemma}

Suppose $A$ and $B_1, \ldots, B_m$ are compact, convex sets in $\R^n$ with nonempty interior,
and that $X$ is uniformly distributed on $A$ while each $Y_i$ is uniformly distributed on $B_i$. 
Applying Proposition~\ref{prop:vol-ent}, we have that
\ben\begin{split}
\log\bigg[ \frac{|A+\sum_{i\in [m]} B_i|}{|A|} \bigg] -n\log (1+m)
&\leq h\bigg(X+ \sum_{i\in[m]} Y_i \bigg) - h(X) \\
&\leq \frac{1}{\binom{m-1}{k-1}} \sum_{\setS\in\collS_k}  \bigg[ h\bigg( X+\sum_{i\in\setS} Y_{i}\bigg) - h(X) \bigg] \\
&\leq \frac{1}{\binom{m-1}{k-1}} \sum_{\setS\in\collS_k}  \log\bigg[ \frac{|A+\sum_{i\in \setS} B_i|}{|A|} \bigg] .
\end{split}\een
Thus we obtain the following corollary.

\begin{corollary}\label{cor:plunnecke}
Let $\collS_k$ denote the collection of all subsets of $[m]=\{1,\ldots, m\}$ that are of cardinality $k$.
Let $A$ and $B_{1},\ldots, B_{m}$ be convex bodies in $\R^{n}$, and suppose
$$
\bigg|A+\sum_{i\in \setS} B_i \bigg|^{\nth} \leq c_{\setS} |A|^{\nth}
$$
for each $\setS\in\collS_k$, with given numbers $c_{\setS}$.
Then
$$
\bigg|A+\sum_{i\in [m]} B_i \bigg|^{\nth}
\leq (1+m) \bigg[\prod_{\setS\in\collS_k} c_{\setS}\bigg]^{\frac{1}{\binom{m-1}{k-1}}} |A|^{\nth} .
$$
\end{corollary}

In particular, by choosing $k=1$, one already obtains an interesting inequality for volumes of Minkowski sums:
for convex bodies, if $|A+B_i|^{\nth} \leq c_{i} |A|^{\nth}$ for each $i=1,\ldots, m$, then
$$
\bigg|A+\sum_{i\in [m]} B_i \bigg|^{\nth}
\leq (1+m) \bigg[\prod_{i\in[m]} c_i\bigg] \, |A|^{\nth} .
$$

Inequalities of this type are well known for set cardinalities in the context of {\it finite} subsets of groups.
In fact, they are important inequalities in the field of 
additive combinatorics, where they are called Pl\"unnecke-Ruzsa inequalities (see, e.g., the book of
T.~Tao and V.~Vu \cite{TV06:book}). These were introduced by H.~Pl\"unnecke \cite{Plu70} and
generalized with a simpler proof by I.~Ruzsa \cite{Ruz89}; a more recent generalization is proved 
in \cite{GMR08}, and entropic versions are developed in \cite{MMT11}. 
For illustration, the form of Pl\"unnecke's inequality developed in  \cite{Ruz89} states that
if $A, B_1,\ldots ,B_k$ are finite sets in a commutative group and $|A| =m, |A + B_i| = \alpha_i m$, for $1 \leq i \leq k$,
then there exists an $X \subset A, X \neq \phi$ such that
\ben
|X+B_1+\ldots +B_k| \leq \alpha_1 \ldots \alpha_k |X| .
\een
Thus one may think of Corollary~\ref{cor:plunnecke} as providing continuous analogues of the Pl\"unnecke-Ruzsa inequalities
in the context of volumes of convex bodies in Euclidean spaces, where going from the discrete to the continuous
incurs the extra factor of $(1+m)$, but one does not need to bother with taking subsets of the set $A$.

Let us note that T.~Tao \cite{Tao08:1} has previously developed a continuous analogue
of Freiman's theorem, which is related to the Pl\"unnecke-Ruzsa inequalities. Specifically, 
\cite[Proposition~7.1]{Tao08:1} asserts that if
$A$ is an open bounded non-empty subset of $\R^n$ such that $|A+A| \leq K |A|$ 
for some $K \geq 2^n$, then
there exists an $\epsilon> 0$ and a set $P$ which is the sum of $O_K(1)$ arithmetic
progressions in $\R^n$ such that $A\subset P+B(0,\epsilon)$ and $|P+B(0,\epsilon)|\approx_K |A|$.
However, this kind of continuous analogue is different in nature from the one we propose above,
since it focuses on algebraic rather than convex structure. 
Another notable continuous analogue of Freiman's theorem is developed in
the more general context of locally compact, abelian groups by T.~Sanders \cite{San09}.


\section{The log-concave case}
\label{sec:lcpf}
\setcounter{equation}{0}

In the log-concave case Theorems~\ref{thm:repi}--\ref{thm:repi2}  are somewhat simpler due to the
property that the class of log-concave probability densities is closed
under the convolution operation.

Let us describe the argument, assuming that $X$ and $Y$ have log-concave
densities, say, $f$ and $g$, respectively. 
First consider the case, where both $f$ and $g$ are even functions
in the sense that $f(x)=f(-x)$ and $g(x)=g(-x)$.

\begin{proof}{\bf (of Theorem~\ref{thm:repi} in the symmetric log-concave case.)}
In this case, the essential supports 
$$
K_f = \{f(x) \geq c_0^n\, \|f\|\} \quad\text{and}\quad
K_g = \{g(x) \geq c_0^n\, \|g\|\} ,
$$
where $c_0 \in (0,1)$ is a universal constant, are symmetric convex sets.
By Corollary~\ref{cor:innerprod}
, one may bound the entropy power as follows:
\bee
H(X+Y) & \leq & e^2 \bigg(\int f(x) g(x)\,dx\bigg)^{\!-2/n} \\ 
       & \leq & 
e^2  c_0^{-4}\, \|f\|^{-2/n} \|g\|^{-2/n}\, |K_f \cap K_g|^{-2/n}.
\ene
Moreover, if both $K_f$ and $K_g$ are in $M$-position, which may be assumed, 
then we have by deploying Corollary~\ref{cor:M-sum} and relation \eqref{eq:5n} that
$$
|K_f \cap K_g|^{1/n} \geq c_1 \min\{|K_f|^{1/n},|K_g|^{1/n}\} \geq
\frac{c_1}{2}\, \min\left\{\|f\|^{-1/n},\|g\|^{-1/n}\right\}.
$$
Hence, with some numerical constant $C > 0$
$$
H(X+Y) \leq C\,\max\left\{\|f\|^{-2/n},\|g\|^{-2/n}\right\} \leq C\,
\max\{H(X),H(Y)\},
$$
where on the last step we made use of the general relation 
$H(X) \geq \|f\|^{-2/n}$. This proves Theorem~\ref{thm:repi} (and therefore Theorem~\ref{thm:repi2})
in the symmetric log-concave case.
\end{proof}

In the general non-symmetric case one may use the inequality \eqref{eq:M-sum2} for 
non-symmetric sets in $M$-position. There is also
another argument based on the following elementary observation.

\begin{lemma}\label{lem:renyi}
For any log-concave probability density $f$
on $\R^n$,
\be\label{eq:renyi}
2^{-n}\, \|f\|\, \leq\,
\int f(x)^2\,dx\, \leq\, \|f\|.
\en
\end{lemma}

The right inequality is trivial and holds without any assumption on the
density. To derive the left inequality, write the definition of the 
log-concavity,
$$
f(tx + sy) \geq f(x)^t\, f(y)^s, \quad x,y \in \R^n, \ t,s> 0, t+s = 1.
$$
It may also be applied to $f^{1/t}$, so
$f(tx + sy)^{1/t} \geq f(x)\, f(y)^{s/t}$.
Integrating with respect to $x$ and using the assumption that $\int f = 1$, 
we get
$$
t^{-n} \int f(x)^{1/t}\,dx \geq f(y)^{s/t}.
$$
It remains to optimize over $y$'s, so that
$\int f(x)^{1/t}\,dx \geq t^{n}\, \|f\|^{s/t}$,
and then take the values $t = s = 1/2$.

\begin{proof}{\bf (of Theorem~\ref{thm:repi} in the general log-concave case.)} 
One may use symmetrization. Let $X$ be a random vector in $\R^n$ with a log-concave 
density $f$. Let $X'$ be an independent copy of $X$, thus
with density $\tilde f(x) = f(-x)$. Then the random vector
$X'' = X-X'$ has a symmetric log-concave distribution with density
$$
f * \tilde f(x) = \int f(x+y) f(y)\,dy,
$$
whose norm satisfies, by \eqref{eq:renyi},
\be\label{eq:renyi-applied}
\|f\|\, \geq \|f * \tilde f\| = f * \tilde f(0) \geq 2^{-n}\, \|f\|.
\en

Now, let's do the same symmetrization with another log-concave random vector 
$Y$ in $\R^n$ with density $g$, assuming that it is independent 
of $X$. Then we are in position to apply to $(X'',Y'')$ 
the symmetric part of Theorem~\ref{thm:repi}, which gives
\be\label{eq:symm-pf}
H(u_1(X'') + u_2(Y'')) \leq C\, (H(X'') + H(Y'')),
\en
for some linear volume preserving map $u_i:\R^n \rightarrow \R^n$
and some universal constant $C$. 

But since the entropy power may only increase when adding to a given random
vector an independent summand, the left side of \eqref{eq:symm-pf} is 
greater than or equal to $H(u_1(X) + u_2(Y))$. On the other hand,
by Corollary~\ref{cor:lc-maxnorm}
 and applying \eqref{eq:renyi-applied}, we have
$$
H(X'') \leq e^2\, \big\|f * \tilde f\big\|^{-2/n} \leq 
2e^2\, \|f\|^{-2/n} \leq 2e^2\,H(X).
$$
With a similar bound for the random vector $Y$, we arrive at
$$
H(u_1(X) + u_2(Y))  \leq 2e^2C\,(H(X) + H(Y)).
$$
\end{proof}


\section{Proof of Theorem~\ref{thm:repi}}
\label{sec:genpf}
\setcounter{equation}{0}

In order to involve in Theorem~\ref{thm:repi} more general convex measures,
we need to apply the more delicate Propositions~\ref{prop:cvx-maxnorm}, \ref{prop:cvx-aep} and \ref{prop:M-cvx}.
Moreover, since the previous argument based on the log-concavity of the
convolution of two log-concave densities has no extension to the
class of convex measures (with negative convexity parameter $\k$), 
we have to appeal to the submodularity property of the entropy functional.

Throughout this section let $Z$ denote a random vector in $\R^n$ uniformly
distributed in the Euclidean ball $D$ with center at the origin
and volume one. In particular, $h(Z)=0$, and by Proposition~\ref{prop:submod},
\be\label{eq:submod-again}
h(X+Y) \leq h(X+Z)\, + \,h(Y+Z),
\en
for all random vectors $X$ and $Y$ in $\R^n$ that are independent of each 
other and of $Z$ (provided that all entropy powers are well-defined).

Let $X$ and $Y$ have densities of the form \eqref{eq:cvx-def}. In view of the homogeneity 
of the inequality \eqref{eq:repi} of Theorem~\ref{thm:repi}, we may assume that $\|f\| \geq 1$ and 
$\|g\| \geq 1$. Then, by \eqref{eq:submod-again}, our task reduces to showing that both $h(X+Z)$ 
and $h(Y+Z)$ can be bounded from above by quantities, depending on $\beta_0$, 
only (under further assumption on $\beta_0$). This can be achieved by putting
the distributions of $X$ and $Y$ in $M$-position. 

Thus, what we need is:

\begin{lemma}\label{lem:ent-M}
Let $X$ be a random vector in $\R^n$ independent of $Z$
with density $f = V^{-\beta}$ such that $\|f\| \geq 1$, where $V$ is a convex 
function, and where $\beta$ is in the range 
\be\label{eq:range-b}
\beta \geq \max\{ 2n+1, \beta_0 n\} \ (\beta_0 > 2).
\en
Then for some linear volume preserving map $u:\R^n \rightarrow \R^n$, we have
$
H(u(X)+Z) \leq C_{\beta_0}
$
with constants depending on $\beta_0$, only.
\end{lemma}

\vskip5mm
\begin{proof}
By Proposition~\ref{prop:M-cvx}, for some affine volume preserving map 
$u:\R^n \rightarrow \R^n$, the distribution $\widetilde \mu$ of 
$\widetilde X = u(X)$ satisfies
$$
\widetilde \mu(D)^{1/n} \geq c_0
$$
with a numerical constant $c_0 > 0$ (which does not depend on $\beta_0$, 
since $\beta_0$ is well separated from 1). Let $\tilde f$ denote the density of 
$\widetilde X = u(X)$. Then the density $p$ of $S = \widetilde X + Z$,
given by
$$
p(x) = \int_D \tilde f(x-z)\,dz = \widetilde \mu(D - x),
$$
satisfies
\be\label{eq:8.interim}
\|p\| \geq p(0) \geq c_0^n.
\en
Hence, in order to bound the entropy power $H(S)$, it will
be sufficient to know the convexity parameter of the distribution of 
$S$. (Here is the place where the conditions \eqref{eq:range-b} arise).

As we know from the Borell characterization, the distribution 
$\widetilde \mu$ of $\widetilde X$ is $\k'$-concave with the convexity parameter
$$
\k' = - \frac{1}{\beta - n}.
$$
Also, recall that $Z$ has the uniform distribution in $D$ with the
parameter $\k'' = \frac{1}{n}$. In order to judge about convexity properties 
of the convolution $p = \widetilde f * g$, where $g = 1_D$ is the density of 
$Z$, one may apply Proposition~\ref{prop:cvx-conv}. Then we need to check the condition
$$
\k' + \k'' > 0,
$$
which in our case is equivalent to $\beta > 2n$. By \eqref{eq:range-b}, this requirement
is met, so $S$ has a $\k$-concave distribution with parameter 
$\k$ given by
$$
\frac{1}{\k} = \frac{1}{\k'} + \frac{1}{\k''} = - (\beta - 2n),
$$
that is, with $\k =  - \frac{1}{\beta - 2n}$. Equivalently, $S$ 
has a density of the form $p = W^{-\beta_{S}}$ for some convex 
function $W$ and with the $\beta$-parameter
$$
\beta_{S} = n - \frac{1}{\k} = n + (\beta - 2n) = \beta - n.
$$

We can now apply Proposition~\ref{prop:cvx-maxnorm}
 to the random vector $S$. Together 
with \eqref{eq:8.interim} it gives
$$
H(S) \leq C\, \|p\|^{-2/n} \leq C \cdot c_0^{-2},
$$
provided that $\beta_S \geq n+1$, $\beta_S \geq \beta_0'\, n$, $\beta_0' > 1$, 
and with constants depending on $\beta_0'$. With $\beta_0' = \beta_0 - 1$,
these conditions are equivalent to \eqref{eq:range-b}.

Lemma~\ref{lem:ent-M} and therefore Theorem~\ref{thm:repi} are proved.
\end{proof}

It would be interesting to explore the range of $\beta$,
such that the inequality of Theorem~\ref{thm:repi} holds true with
$\beta$-dependent constants. On the other hand, the following
statement (proved in \cite{BM11:nonunif}) is true:

\begin{proposition}\label{prop:counter} 
For any constant $C$, there is a convex
probability measure $\mu$ on the real line with the following property. If $X$ and $Y$ are independent random variables distributed according to $\mu$, then
$\min(H(X+Y),H(X-Y)) \geq CH(X)$.
\end{proposition} 

In other words, Theorem~\ref{thm:repi} does not hold with an absolute
constant to serve for the entire class of convex measures
(already in dimension one).


\section{Discussion}
\label{sec:disc}
\setcounter{equation}{0}

One may wonder how to find specific positions 
(that is, the linear maps $u_1$ and $u_2$) for the distributions of the 
random vectors $X$ and $Y$ in Theorem~\ref{thm:repi}. Natural candidates are the
so-called isotropic positions.

Let us recall the well-known and elementary fact that, in the class of 
all (absolutely continuous) probability distributions on $\R^n$ 
with a fixed covariance matrix, the entropy $h(X)$ is maximized when $X$ has 
a normal distribution. Equivalently, for any affine volume preserving map
$T$ of the space $\R^n$,
\be\label{eq:maxent}
\frac{1}{2\pi e}\, H(X) \leq \int \frac{|Tx|^2}{n}\, f(x)\, dx,
\en
where $f$ is density of $X$. If the right side of \eqref{eq:maxent} is minimized
for the identity map $T(x)=x$, then one says that the distribution of $X$ 
is isotropic or in isotropic position (cf. \cite{MP89}). This is equivalent to 
the property that $X$ has mean at the origin and, for any unit vector $\theta$,
$$
L_f^2 = \|f\|^{2/n} \int \left<x,\theta\right>^2\,f(x)\,dx,
$$
for some number $L_f > 0$, called the {\it isotropic constant} of $f$. If $X$ is
uniformly distributed in a convex body $K$, the number $L_f = L_K$ is
called the isotropic constant of $K$.

Thus, for any random vector $X$ in $\R^n$ with density $f$ regardless of 
whether its distribution is isotropic or not, \eqref{eq:maxent} may be rewritten as
\be\label{eq:isotr}
\frac{1}{2\pi e}\, H(X) \leq L_f^2\, \|f\|^{-2/n}.
\en
In view of the general bound $H(X) \geq \|f\|^{-2/n}$, the above estimate
implies, in particular, that $L_f^2 \geq \frac{1}{2\pi e}$, so the isotropic
constants are separated from zero.

Restricting ourselves to (isotropic) log-concave probability distributions, 
the question of whether the isotropic constants are bounded from above 
by a dimension-free constant is equivalent to the (still open) hyperplane 
problem raised by J. Bourgain in the mid 1980's. As was shown by K. Ball \cite{Bal88}, 
it does not matter whether this problem is stated for the class of (all) 
convex bodies or for the class of (all) log-concave distributions; 
see also \cite{Bob10} for an extension to the class of convex measures.
An affirmative solution of the hyperplane problem is known
for some subclasses of log-concave distributions. For example, $L_f$
is bounded by a universal constant, if the distribution of $X$ is log-concave
and symmetric about the coordinates axes.

Anyhow, the inequalities \eqref{eq:maxent}--\eqref{eq:isotr} suggest the following variant of the 
reverse Brunn-Minkowski inequality. Let $X$ and $Y$ be independent random 
vectors with log-concave densities $f$ and $g$, respectively. 
Applying \eqref{eq:maxent} to $X+Y$ with $Tx = x-x_0$, where $x_0 = \E\, (X+Y)$, 
we obtain that
\be\label{eq:variant}
\frac{1}{2\pi e}\, H(X+Y)\, \leq \,
\frac{1}{n} \int |x|^2\, f(x)\, dx + \frac{1}{n} \int |x|^2\, g(x) dx.
\en
Here the right side is sharpened, when the distributions of $X$ and $Y$
are put in the isotropic position, and then we arrive at
\be\label{eq:balltype1}
\frac{1}{2\pi e}\, H(\widetilde X + \widetilde Y)\, \leq \, 
L_f^2\, H(X) + L_g^2\, H(Y),
\en
where $\widetilde X = u_1(X)$, $\widetilde Y = u_2(Y)$, and where affine volume
preserving maps $u_i$'s are chosen so that both $\widetilde X$ and $\widetilde Y$ 
are isotropic. (Such maps are easily described in terms of the covariance
matrices of $X$ and $Y$).

In particular, if $X$ and $Y$ are uniformly distributed in convex bodies 
$A$ and $B$, respectively, the inequalities \eqref{eq:variant}--\eqref{eq:balltype1} together with the 
lower bound in \eqref{eq:volsum} yield 
$$
\frac{1}{8\pi e}\, |A+B|^{2/n}\, \leq \,
\frac{1}{n |A|} \int_A |x|^2\, dx + \frac{1}{n |B|} \int_B |x|^2\, dx .
$$
In particular, one obtains the following corollary.

\begin{corollary}\label{cor:ball}
Suppose $A$ and $B$ are convex bodies, and $\widetilde A = u_1(A)$ and $\widetilde B = u_2(B)$ are the
bodies after being put in isotropic position. Then
$$
\frac{1}{8\pi e}\, \big|\widetilde A + \widetilde B\big|^{2/n}\, \leq \,
L_A^2\, |A|^{2/n} + L_B^2\, |B|^{2/n} .
$$
\end{corollary}

Therefore, if the isotropic constants $L_A$ and $L_B$ are known 
to be bounded by a constant, say $C_0$, then Corollary~\ref{cor:ball} 
provides a reverse Brunn-Minkowski 
inequality \eqref{eq:reverseBM} with $C = C_0 \sqrt{8\pi e}$.

A result such as Corollary~\ref{cor:ball} was first obtained, using a different argument, 
by K. Ball in his thesis \cite{Bal86:phd}.

\bibliographystyle{plain}

\end{document}